\documentclass[11 pt, leqno]{amsart}

\usepackage{dsfont}
\usepackage{enumerate}
\usepackage{amsfonts}
\usepackage{amsthm}
\usepackage{amssymb}
\usepackage{latexsym}
\usepackage{amsmath}
\usepackage{mathrsfs}
\usepackage{graphicx}
\usepackage{color}
\usepackage{hyperref}

\usepackage{a4wide}

\usepackage{comment}

\theoremstyle{plain}
\newtheorem{tw}{Theorem}[section]

\newtheorem {lem} [tw]{Lemma}
\newtheorem {prop}[tw] {Proposition}

\newtheorem{cor}[tw]{Corollary}
\newtheorem{con}[tw]{Conjecture}
\newtheorem*{mainthm}{Theorem}

\theoremstyle{definition}
\newtheorem {deft}[tw] {Definition}
\newtheorem {rem} [tw]{Remark}

\newcommand{\cst}{\ifmmode\mathrm{C}^*\else{$\mathrm{C}^*$}\fi}

\newcommand{\bc} {\Bbb C}
\newcommand{\bn}{\Bbb N}
\newcommand{\br}{\Bbb R}

\newcommand{\alg} {\mathsf{A}}

\newcommand{\Hil}{\mathsf{H}}

\newcommand{\blg}{\mathsf{B}}

\newcommand{\oned}{\{1,\ldots,d\}}

\newcommand{\FockqH}{\mathcal{F}_q(\Hil_U)}
\newcommand{\FockqHalg}{\FockqH_{\textup{alg}}}

\newcommand{\mlg}{\mathsf{M}}
\newcommand{\cross}{\textup{cross}}

\newcommand{\inv}{\textup{inv}}

\newcommand{\ot}{\otimes}

\numberwithin{equation}{section}

\begin{document}

%\usetikzlibrary{arrows,positioning}

\author{Manish Kumar}  
\address{Institute of Mathematics of the Polish Academy of Sciences,
	ul.~\'Sniadeckich 8, 00--656 Warszawa, Poland}
\email{mkumar@impan.pl}
\author{Adam Skalski}
%\address{Institute of Mathematics of the Polish Academy of Sciences,
%	ul.~\'Sniadeckich 8, 00--656 Warszawa, Poland}
\email{a.skalski@impan.pl}
\author{Mateusz Wasilewski}
%\address{Institute of Mathematics of the Polish Academy of Sciences,
%	ul.~\'Sniadeckich 8, 00--656 Warszawa, Poland}
\email{mwasilewski@impan.pl}

\title{Full solution of the factoriality question for  $q$-Araki-Woods von Neumann algebras via conjugate variables}

\begin{abstract}
We establish factoriality  of $q$-Araki-Woods von Neumann algebras (with the number of generators at least two)  in full generality, exploiting the approach via conjugate variables developed recently in the tracial case by Akihiro Miyagawa and Roland Speicher, and abstract results of Brent Nelson. We also establish non-injectivity and determine the type of the factors in question. The factors are solid and full when the number of generators is finite.
\end{abstract}



\maketitle

The history of $q$-deformations in the context of (Gaussian) von Neumann algebras started with the paper \cite{BKS}, where the authors defined the so-called $q$-von Neumann algebras $\Gamma_q(\Hil_\br)$, with $q\in (-1,1)$, acting on the $q$-deformed Fock space associated with the complexification of a real Hilbert space $\Hil_\br$. These algebras are tracial, and could be seen as natural deformations of the free group factors. Thus from the beginning it was a natural question whether they also have trivial centres (i.e.\ are factors) -- naturally excluding the case of $\dim(\Hil_\br) =1$, when $\Gamma_q(\Hil_\br)$ is generated by a single self-adjoint operator. The original paper showed that this is indeed true for infinite-dimensional $\Hil_\br$. It took almost ten years to establish factoriality in full generality, that is for $\dim(\Hil_\br) \geqslant 2$. This was achieved by \'E. Ricard in \cite{Eric}, after earlier partial results obtained in \cite{Sniady} and \cite{Ilona}.

Meanwhile in \cite{Hiai} another von Neumann algebraic $q$-deformation was constructed by F.\,Hiai, who combined the construction of \cite{BKS} with the quasi-free (Araki-Woods) deformation of D. Shlyakhtenko (\cite{Dima}). This time the initial data involves not only the real Hilbert space $\Hil_\br$, but also a group $(U_t)_{t \in \br}$ of orthogonal transformations of $\Hil_\br$, and the resulting von Neumann algebras $\Gamma_q(\Hil_\br, U_t)$ are in general non-tracial (see the formal definition in the beginning of the next section). 
Also here the factoriality question has attracted a lot of interest, starting from the original article \cite{Hiai}, and later continued in \cite{BM}, \cite{SW} and \cite{BMRW}. In the last of these P.\,Bikram, K.\,Mukherjee,  \'E. Ricard and S.\,Wang establish factoriality 
whenever $\dim(\Hil_\br) \geqslant 3$ and also for $\dim(\Hil_\br) = 2$, but only when the deformation generated by $(U_t)_{t \in \br}$ is sufficiently large. It should be also noted that \cite{Brent0} shows that for finite dimensional $\Hil_\br$ and  $|q|$ small enough $\Gamma_q(\Hil_\br, U_t)$  is isomorphic to $\Gamma_0(\Hil_\br, U_t)$, so in that case already the results of \cite{Dima} yield factoriality. 

Here we resolve the matter in full, showing that for arbitrary $q\in (-1,1)$ the algebra $\Gamma_q(\Hil_\br, U_t)$ is a factor if and only if $\dim(\Hil_\br) \geq2$. We take a very different approach to that taken in \cite{BM} and \cite{BMRW}, and instead of studying `mixing subspaces' of $\Gamma_q(\Hil_\br, U_t)$ we adopt a very recent approach to  `dual' or `conjugate variables' developed in the tracial case of $\Gamma_q(\Hil_\br)$ in \cite{MS} by A.\,Miyagawa and R.\,Speicher. It turns out that the construction of \cite{MS} can be also extended to the non-tracial case, which in conjunction with the general study of consequences of the existence of suitable generators with finite free Fisher information conducted in \cite{Brent} allows us not only to settle completely the factoriality question, but also establish the non-injectivity and determine the type of the associated factor. Thus we obtain the following theorem, which is the main result of this paper.

\begin{mainthm}
Let $q \in (-1,1)$, let $\Hil_\br$ be a real Hilbert space of dimension at least $2$, and let $(U_t)_{t \in \br}$ be a group of orthogonal transformations of $\Hil_\br$.
Then the $q$-Araki-Woods von Neumann algebra $\Gamma_q(\Hil_\br, U_t)$ is a non-injective factor of type  
\[
\left\{\begin{array}{ll} \mathrm{III}_1 & \text{ if } G=\mathbb{R}^{\times}_{\ast}, \\ \mathrm{III}_{\lambda} & \text{ if } G= \lambda^{\mathbb{Z}}, 0<\lambda<1, \\ \mathrm{II}_1 & \text{ if } G=\{1\},  \end{array}   \right.
\]	
where $G< \mathbb{R}^{\times}_{\ast}$ is the closed subgroup generated by the spectrum of the generator of  $(U_t)_{t \in \br}$.	If $\dim (\Hil_\br)< \infty$ then these factors are solid and full.
\end{mainthm}

Note in particular that if the weakly mixing part of $(U_t)_{t \in \br}$ is non-trivial, the factor $\Gamma_q(\Hil_\br, U_t)$ is always of type $\mathrm{III}_1$, as shown already in \cite[Theorem 6.1]{BM}.

The plan of the paper is as follows: in the remainder of the introduction we set some notation. In Section \ref{sec:dc}, treating solely the case of finite-dimensional initial space, we introduce the notions of dual/conjugate variables and establish their existence for $q$-Gaussian systems inside $q$-Araki-Woods algebras. Then in Section \ref{sec:main} we establish the main results of the paper; in particular Theorem above is a combination of Theorems \ref{thm:factor}, \ref{thm:noninj} and \ref{thm:full}. Finally we record a non-isomorphism result for $q$-Gaussian von Neumann algebras for a respectively finite-dimensional and infinite-dimensional initial Hilbert space and propose a related conjecture in the general case.

Throughout the paper we fix $q \in (-1,1)$. All scalar products are linear on the right. Given a real Hilbert
space $\Hil_\br$ and $(U_t)_{t \in \br}$,  a group of orthogonal transformations of $\Hil_\br$, we denote  the associated $q$-Araki-Woods von Neumann algebra by $\Gamma_q(\Hil_\br, U_t)$.  For a detailed description of the construction of  $\Gamma_q(\Hil_\br, U_t)$ we refer to the original article \cite{Hiai}.

\section{Dual and conjugate variables for $q$-Araki-Woods von Neumann algebras in finite dimensions}

\label{sec:dc}
%\subsection*{General setup}

In this section we will consider only the case of finite-dimensional $\Hil_\br$.
Fix then $d \in \bn$ and write $\Hil$ for the complexification of $\Hil_\br=\br^d$. We assume that we are also given 
$(U_t)_{t \in \br}$, a group of orthogonal transformations of $\br^d$, whose generator (both on $\Hil_\br$ and on $\Hil$) will be denoted by $A$. The space $\Hil$ is thus equipped both with the standard scalar product and with the deformed scalar product $\langle \xi, \eta \rangle_U:= \langle \xi, \frac{2A}{1+A} \eta \rangle$; if we want to stress the difference we will sometimes use the notation $\Hil_U$. Further we write $\FockqH$ for the associated $q$-Fock space, %(which should be more precisely denoted $\mathcal{F}_q(\Hil_U)$),
with $\FockqHalg$ as the subspace spanned by finite tensors, and $e_0=\Omega$ the vacuum vector.  Given $\xi \in \Hil_U$ we can consider the associated left creation operator $l^*(\xi) \in B(\FockqH)$, and further the relevant Gaussian operator by $s(\xi):=l^*(\xi) + (l^*(\xi))^*$. 
We denote the associated $q$-Araki-Woods von Neumann algebra  $\Gamma_q(\Hil_\br, U_t):=\{s(\xi): \xi \in \Hil_\br\}''$  simply by $\mlg$ and the canonical $q$-quasi free state $\langle\Omega,\cdot\Omega\rangle_{\mathcal{F}_q(H_U)}$ by $\varphi$. Finally for each $\xi\in \FockqHalg$, we denote by $W(\xi)$ the unique element in $\mlg$  which satisfies $W(\xi)\Omega=\xi$.

 Let $\{e_1,\ldots, e_d\}$ be a linearly independent set of vectors in $\Hil$, and for $i \in \oned$  let $A_i=W(e_i)$, i.e.\ $A_i \in \mlg$ and $A_i\Omega=e_i$.  We say that a tuple $(D_1,\ldots, D_d)$ of unbounded operators on $\mathcal{F}_q(H_U)$ with $\FockqHalg$ contained in their domains and $\mathds{1}$ contained in the domains of their adjoints is a {\em (normalized) dual system} for $(A_1,\ldots, A_d)$ if for all $i, j \in \oned$ \[{[D_i,A_j]}=\langle\bar{e}_j, e_i\rangle_U P_{\bc \Omega}=\varphi(A_jA_i)P_{\bc \Omega}\;\;\mbox{ and } D_i\Omega=0.\]
Here $\bar{\xi}$ denotes the usual conjugate of a vector $\xi$ in $\mathbb{C}^d$ and $P_{\mathbb{C}\Omega}$ denotes the projection onto the one-dimensional subspace $\mathbb{C}\Omega$.  Before we proceed any further, let us note that existence of dual variables implies existence of conjugate variables. We will actually show directly that the existence of dual variables implies existence of the conjugate variables with respect to the quasi-free difference quotients (see \cite[Definition 3.11]{Brent}), which in turn implies existence of the usual conjugate variables (see \cite[Remark 3.13]{Brent}). Note also here that as in \cite{Brent} we do not require that individual $A_i$ are self-adjoint -- although in  Corollary \ref{cor:freeFisher} below we will be concerned with the situation where they form a self-adjoint set.

Recall that the \emph{quasi-free difference quotients} $\partial_{i}$ are defined as unique derivations from $\mathbb{C}[A_i,\dots, A_d]$ into $\mlg \overline{\otimes} \mlg^{op}$ such that $\partial_i(A_j) := \varphi(A_j A_i)\mathds{1}\otimes \mathds{1}$ for all $i,j\in \oned$. The \emph{conjugate variable} for $\partial_i$ will be a vector $\xi_i \in L^{2}(\mlg,\varphi)$ such that 
\[
\langle \xi_i, x\mathds{1}\rangle = \langle \mathds{1}\otimes \mathds{1}, \partial_i(x)(\mathds{1}\otimes \mathds{1})\rangle 
\]
for all $x \in \mathrm{dom}(\partial_i)$.

\begin{prop}[{See \cite[Theorem 2.5]{MS}}]
Suppose that $(D_1,\dots, D_d)$ is a normalized dual system for $(A_1,\dots, A_d)$. Then $(D_{1}^{\ast}\mathds{1},\dots, D_{d}^{\ast}\mathds{1})$ are conjugate variables for $(A_1,\dots, A_d)$.
\end{prop}
\begin{proof}
It suffices to check that for all $i \in \oned$, $n \in \bn$ and $j_1, \ldots,j_n \in \oned$ we have $\langle D_{i}^{\ast}\mathds{1}, A_{j_{1}}\dots A_{j_{n}}\mathds{1}\rangle = \langle \mathds{1}\otimes \mathds{1}, \partial_i(A_{j_{1}}\dots A_{j_{n}})(\mathds{1}\otimes \mathds{1})\rangle$. The left-hand side is equal to $\langle \mathds{1}, D_{i} A_{j_{1}}\dots A_{j_{n}}\mathds{1}\rangle$. The defining property of $D_i$ says that $D_i A_{j_1} = A_{j_1} D_i + \varphi(A_{j_{1}} A_{i}) P_{\Omega}$. It follows that 
\begin{align*}
\langle \mathds{1}, D_{i} A_{j_{1}}\dots A_{j_{n}}\mathds{1}\rangle &= \langle \mathds{1}, A_{j_{1}} D_{i}\dots A_{j_{n}}\mathds{1}\rangle + \varphi(A_{j_{1}} A_{i}) \langle \mathds{1}, A_{j_{2}}\dots A_{j_{n}}\mathds{1}\rangle \\
&= \langle \mathds{1}, A_{j_{1}} D_{i}\dots A_{j_{n}}\mathds{1}\rangle + \varphi(A_{j_{1}} A_{i}) \varphi( A_{j_{2}}\dots A_{j_{n}})
\end{align*}
Continuing in this way we will obtain the final formula:
\[
\langle \mathds{1}, D_{i} A_{j_{1}}\dots A_{j_{n}}\mathds{1}\rangle = \sum_{k=1}^{n}\varphi(A_{j_{k}} A_{i}) \varphi(A_{j_{1}}\dots A_{j_{k-1}})\varphi(A_{j_{k+1}}\dots A_{j_{n}}) + \langle \mathds{1}, A_{j_{1}}\dots A_{j_{n}}D_i \mathds{1}\rangle,
\]
where the last term vanishes as $D_i \mathds{1}=0$. This is equal to $\langle \mathds{1}\otimes \mathds{1}, \partial_i(A_{j_1}\dots A_{j_{n}})(\mathds{1}\otimes \mathds{1})\rangle$, because the value $\partial_i(A_{j_1}\dots A_{j_{n}})$ can be computed exactly as for free difference quotients, merely replacing Kronecker deltas $\delta_{i j_k}$ with the covariance $\varphi(A_{j_{k}} A_i)$.
\end{proof}

\begin{lem}\label{lem:basechange}
	Let $\{e_i\}_{1\leqslant i\leqslant d}$ and $\{f_i\}_{1\leqslant i\leqslant d}$ be two linearly independent sets in $\Hil$ such that for every $j \in \oned$ we have $f_j=\sum_{k=1}^dx_{jk}e_k$ for some $x_{jk}\in \mathbb{C}$. If $A_i=W(e_i)$ and $C_i=W(f_i)$, $i \in \oned$, then  a dual system for $\{A_i\}_{1\leqslant i\leqslant d}$ exists if and only if one for $\{C_{i}\}_{1\leqslant i\leqslant d}$ does. 
\end{lem}
\begin{proof}
	Note that the definition of Wick operators assures that $C_j=\sum_{k=1}^d{x_{jk}}A_k$, $j \in \oned$.    If $\{D_i\}_{1\leqslant i\leqslant d}$ denotes the dual system for $\{A_i\}_{1\leqslant i\leqslant d}$, then $\{E_i\}_{1\leqslant i\leqslant d}$ is the dual system for $\{C_i\}_{1\leqslant i\leqslant d}$, where for each $i \in \oned$ we set $E_i=\sum_{k=1}^dx_{ik}D_k.$
	Indeed, let us check:
\begin{align*} [E_i, C_j] &= \sum_{k=1}^d \sum_{l=1}^d x_{ik} x_{jl} [D_k, A_l] = 
\sum_{k=1}^d \sum_{l=1}^d x_{ik} x_{jl} \langle\bar{e}_l, e_k\rangle_U  P_{\bc \Omega}\\&=  \langle\overline{ \sum_{l=1}^d x_{jl}e_l}, \sum_{k=1}^dx_{ik}  e_k\rangle_U P_{\bc \Omega} = \langle\bar{f}_j, f_i\rangle_U P_{\bc \Omega}. \end{align*}
\end{proof}

\subsection*{Dual variables}
Fix then $\{e_1,\ldots, e_d\}$, an orthonormal set in $\Hil$ with respect to the undeformed scalar product and as before let $A_i=W(e_i)$.
For $i, j \in \oned$ set $B_{ij}=\langle \overline{e}_i,e_j\rangle_U.$ Denote by $[d]^*$ the set of words in letters from the alphabet $\{1,\ldots,d\}$ and for any word $w=j_n\ldots j_1\in [d]^*$, define $e_{j_n\ldots j_1}=e_{j_n}\otimes\cdots\otimes e_{j_1}$.
 One notes that $W(\xi)=l_{\bar{\xi}}+l_{{\xi}}^*$ for any $\xi\in H$ where $l_\xi^*$ is the creation operator; this is a very easy instance of the general Wick product formula (see for example \cite[Proposition 2.12]{ABW}).  Hence we have $A_i(e_{j_n\ldots j_1})=e_{ij_n\ldots j_1}+\sum_{k=1}^nq^{n-k}B_{ij_k} e_{j_{n}\ldots  \hat{j}_k\ldots j_1}$, where $\hat{j}_k$ means to omit $j_k$.

The aim is to exploit the results of \cite{Brent}; to that end we want to first define for each $i \in \oned$ operators
$D_i: \FockqHalg \to \FockqHalg$ such that
\[ D_i \Omega = 0, \;\;\; [D_i, A_j] = B_{ji} P_{\bc \Omega}, \;\;\; j \in \oned.\]
We use below the notation of \cite[Section 4]{MS}, both in the formulation and in the proofs; in particular $B(n+1)$ appearing the following lemma denotes a collection of partitions introduced after \cite[Example 4.3]{MS}. We will just note the places where the arguments need to be extended or modified.

\begin{lem}
	The algebraic formula for the dual variables is given as follows ($i \in \oned$, $n \in \bn$, $j_1, \ldots,j_n \in \oned$):
	\[D_i (e_{j_n} \ldots e_{j_1}) = \sum_{\pi \in B(n+1)} (-1)^{{\pi(0)-1}} q^{\cross(\pi)} \delta_{p(\pi)}^B e_{s(\pi)}, \]
	where $\delta_{p(\pi)}^B:= \prod_{\underset{{l>m}}{(l,m)}\in \pi} B_{j_l,j_m}$.	
\end{lem}

\begin{proof}
	Check first that 
	\[ [D_i, A_j] \Omega =D_i e_j = \sum_{\pi \in B(2)}	(-1)^{\pi(0)-1} q^{\cross(\pi)} \delta_{p(\pi)}^B e_{s(\pi)} = B_{ji} \Omega.\]
	Then we compute 
	\begin{align*}
		D_i A_{j_{n+1}} (e_{j_n} \ldots e_{j_1}) =& D_i \left(e_{j_{n+1} \ldots j_1} + \sum_{l=1}^n q^{n-l} B_{j_{n+1}, j_l}  e_{j_{n} \ldots \hat{j_l} \ldots j_1} \right)
		\\&= D_i e_{j_{n+1} \ldots j_1}  + \sum_{l=1}^n \sum_{\sigma \in B(n)}  (-1)^{\sigma(0)-1} q^{\cross(\sigma)+n-l} \delta_{p(\sigma)}^B B_{j_{n+1},j_l} e_{s(\sigma)}
	\end{align*}
	and 
	\begin{align*}
		A_{j_{n+1}} D_i(e_{j_n} \ldots e_{j_1})  =&  \sum_{\pi \in B(n+1)} (-1)^{\pi(0)-1} q^{\cross(\pi)} \delta_{p(\pi)}^B e_{j_{n+1}s(\pi)}
		\\ &+ \sum_{\pi \in B(n+1)} \sum_{k=1}^{|s(\pi)|} (-1)^{\pi(0)-1} q^{\cross(\pi)+ |s(\pi)|-k} \delta_{p(\pi)}^B B_{j_{n+1}, j_{s(\pi)_k}} e_{s(\pi)\setminus s(\pi)_k}.
	\end{align*}
	In the first step of the proof of \cite[Proposition 4.5]{MS} all terms in the last factor  of the second sum are identified with some terms  in the last factor of the first sum,  by taking a pair $(\pi,k)$ and setting $\sigma \in B(n)$ by removing the singleton $s(\pi)_k$ and putting $l =s(\pi)_k$. Then we just have to observe that 
	\[  \delta_{p(\sigma)}^B B_{j_{n+1},j_l} e_{s(\sigma)} = \delta_{p(\pi)}^B B_{j_{n+1}, j_{s(\pi)_k}} e_{s(\pi)\setminus s(\pi)_k} \]
	(and compute the crossings exactly as in \cite{MS}).
	
	After the subtracting one is left in the first sum with the following terms:
	\[D_i e_{j_{n+1} \ldots j_1}  + \sum_{l=1}^n \sum_{\sigma \in B(n): \sigma(0)\geqslant l}  (-1)^{\sigma(0)-1} q^{\cross(\sigma)+n-l} \delta_{p(\sigma)}^B B_{j_{n+1},j_l} e_{s(\sigma)}\] 
	Now to each pair $(\sigma,l)$ as above we associate $\sigma' \in B(n+2)$ by inserting a `new point' at $l$ and pairing it with $n+1$. Thus the last expression simplifies to (after counting the crossings as in \cite{MS})
	\[D_i e_{j_{n+1} \ldots j_1}   - \sum_{\sigma' \in B(n+2): \sigma'(n+1) \textup{ not a singleton}}  (-1)^{\sigma'(0)-1} q^{\cross(\sigma')} \delta_{p(\sigma')}^B  e_{s(\sigma')};\] 
	note that singletons do not change under this procedure, and $\delta_{p(\sigma)}^B B_{j_{n+1},j_l} = \delta_{p(\sigma')}^B$.
	
	The  rest of the argument is just collecting the terms. 
\end{proof}

The following is the main result of this Section.

\begin{prop} \label{prop:conjugate}
	For each $i \in \oned$ we have $e_0:=\Omega \in \textup{Dom } D_i^*$. Thus $(D_{1}^{\ast}e_0,\dots, D_{d}^{\ast}e_0)$ forms a set of conjugate variables for $(A_1,\dots, A_d)$.
\end{prop}
\begin{proof}
	As in  \cite[Theorem 4.6]{MS} the proof amounts to studying the expression of the form
	\[ \langle e_0, D_i (\sum_{w \in [d]^*} \alpha_w e_w) \rangle_{\mathcal{F}_q(\Hil_U)}, \]
	where the sum is finite (but arbitrary). It is easy to see that it coincides with
	\[(*):=  \sum_{m=1}^\infty \sum_{\;\;\pi \in B(2m), \pi(0)=m\;\;} \sum_{|w|=2m-1} \alpha_w (-1)^{m-1} q^{\cross(\pi)} \delta_{p(\pi), w}^B \]
where $\delta_{p(\pi),w}^B$ is written for $\delta
_{p(\pi)}^B$ in order to make the dependency on $w$ explicit.	Fix for the moment $m\geqslant 1$ and do two things at once: first rewrite each word $w$ of length $2m-1$ as 
	$vjw'$ with $v,w'$ words of length $m-1$ and $j \in \oned$, and second identify each $\pi \in B(2m)$, $\pi(0)=m$ with a permutation $\pi' \in S_{m-1}$ (exactly as in \cite[Theorem 4.6]{MS}). We then have
	\begin{align*} 
		\sum_{\pi \in B(2m), \pi(0)=m\;\;} & \sum_{|w|=2m-1} \alpha_w (-1)^{m-1} q^{\cross(\pi)} \delta_{p(\pi), w}^B
		\\&= (-1)^{m-1} q^{\frac{m(m-1)}{2}} \sum_{\pi' \in S_{m-1}\;\;} \sum_{|v|=m-1} \sum_{j=1}^d \sum_{|w'|=m-1} 
		\alpha_{vjw'}  q^{\inv(\pi')} B_{ji} \delta_{\pi'(v), w}^B, \end{align*}
 where $\delta_{\rho(v), w}^B = \prod_{l=1}^{m-1} B_{v_{\rho(l)} w_l}=\langle \overline{e}_{\rho(v)}, e_w\rangle_{\Hil_U^{\otimes |w|}}$ and $\inv(\pi')$ denotes the number of inversions of the permutation $\pi'\in S_{m-1}$. Further
	\begin{align*} 
		(-1)^{m-1} q^{\frac{m(m-1)}{2}} \sum_{\pi' \in S_{m-1}\;\;}& \sum_{|v|=m-1} \sum_{j=1}^d \sum_{|w'|=m-1} 
		\alpha_{vjw'}  q^{\inv(\pi')} B_{ji} \delta_{\pi'(v), w}^B 
		\\&= (-1)^{m-1} q^{\frac{m(m-1)}{2}} \sum_{|w'|=m-1} 
		\sum_{|v|=m-1} \sum_{j=1}^d  B_{ji} \alpha_{vjw'}  \langle  \overline{e}_v,  e_{w'}\rangle_{\FockqH} 
		\\&=  (-1)^{m-1} q^{\frac{m(m-1)}{2}} \sum_{|v|=m-1} 
		\langle  \overline{e}_v, \sum_{|w'|=m-1} \sum_{j=1}^d  B_{ji} \alpha_{vjw'} e_{w'} \rangle_{\FockqH}. 
	\end{align*}
	So now we fix $v$ of length $m-1$ and look at the vector
	$\sum_{|w'|=m-1} \sum_{j=1}^d  B_{ij} \alpha_{vjw'} e_{w'}$.
	We note that this is nothing but $\sum_{j=1}^d  B_{ij} \tilde{L}_{vj} \left(\sum_{|w'|=m-1}  \alpha_{vjw'} e_{vjw'}\right)$, where $\tilde{L}_{vj}$ denotes the composition of the $m$ relevant undeformed free annihilation operators (acting on $\FockqHalg$) of the form
	$\tilde{L}_{e_k}$ for $k \in \oned$, whose action on $\FockqHalg$ is given simply by 
	\[\tilde{L}_{e_k} (\xi_1 \ot \cdots \otimes\xi_n) = \langle e_k, \xi_1 \rangle \xi_2 \ot \cdots \otimes\xi_n. \]
Naturally we also have 
\[ \sum_{j=1}^d  B_{ij} {\tilde L}_{vj} \left( \sum_{|w'|=m-1}  \alpha_{vjw'} e_{vjw'} \right)= 
\sum_{j=1}^d  B_{ij} \tilde{L}_{vj} \left(\sum_{|w''|=2m-1}  \alpha_{w''} e_{w''}\right)\]	
where we have used the fact that the set $\{e_1,\ldots,e_d\}$ in $\Hil$ is orthonormal in the undeformed scalar product.
Set $T_{i,v}: \FockqHalg \to \FockqHalg$, $T_{i,v}:= \sum_{j=1}^d  B_{ij} \tilde{L}_{vj}$. We need to argue that $T_{i,v}$ is bounded (and estimate its norm).  Consider then 
a free left `undeformed' annihilation operator $\tilde{L}(\xi)$ for $\xi \in \Hil$. Then for any $\eta \in \Hil$ 
we have 
\[ \langle \xi, \eta \rangle = \left\langle \frac{2A}{1+A} \left(\frac{2A}{1+A}\right)^{-1} \xi, \eta \right\rangle =
\left \langle \left(\frac{2A}{1+A}\right)^{-1} \xi, \eta \right\rangle_U  ,\]
so that setting $\tilde{\xi} = (\frac{2A}{1+A})^{-1} \xi$ we see that $\tilde{L}_\xi = L_{\tilde{\xi}}$,
where $L_{\tilde{\xi}}$  denotes the free left annihilation operator on $\FockqHalg$ i.e. $L_{\tilde{\xi}}(\xi_1\otimes\cdots\otimes\xi_n)=\langle \tilde{\xi},\xi_1\rangle_{U}\xi_2\otimes\cdots\otimes\xi_n$. By  \cite[Lemma 2.2]{MS} (and linearity) we have $\|L_{\tilde{\xi}}\|_{B(\FockqH)} \leqslant C\|\tilde{\xi}\|_{\Hil_U}$ (where $C>0$ depends only on $q$). But 
\[ \|\tilde{\xi}\|_{\Hil_U}^2 = \langle \tilde{\xi}, \tilde{\xi} \rangle_U = \langle \xi, (\frac{2A}{1+A})^{-1} \xi\rangle \leqslant D^2 \|\xi\|_{\Hil}^2,\]
where $D:= \|(\frac{2A}{1+A})^{-1} \|^{\frac{1}{2}}$. Thus finally for each $j\in \oned$ we have $\|\tilde{L}_{e_j}\|_{B(\FockqH)}  \leqslant CD$, and setting $B:=\max_{i,j \in \oned} |B_{ij}|$, we obtain for each $v \in [d]^*$, $|v|=m-1$,
\[\|T_{i,v} \|_{B(\FockqH)} \leqslant d B (CD)^m.\]
Then  we obtain the following:
\[(*) \leqslant  \sum_{m=1}^\infty  q^{\frac{m(m-1)}{2}} \sum_{|v|=m-1} \|\overline{e}_v\|_{\FockqH} \|T_{i,v}\| 
\|\sum_{w \in [d]^*}  \alpha_{w} e_{w}\|_{\FockqH}. \]
It is easy to check (as in \cite{MS}) that if we set $E = \max_{i,j \in \oned} |\langle \overline{e}_i, \overline{e}_j \rangle_U|$ then we have for each $v\in [d]^*$ of length $k$ the estimate
\[ \|\overline{e}_v\|^2_{\FockqH} \leqslant E^k [k]_{|q|}!. \]
The rest is just gathering the estimates:
\[ (*) \leqslant \left( \sum_{m=1}^\infty  q^{\frac{m(m-1)}{2}} d^{m-1} E^{\frac{m-1}{2}} \sqrt{[m-1]_{|q|}!} d B (CD)^m \right)
\|\sum_{w \in [d]^*}  \alpha_{w} e_{w}\|_{\FockqH}, \]
and noting that the series inside the brackets converges.
\end{proof}

\begin{rem} We can explicitly write the formula for the conjugate variables for the set $(A_1, \ldots, A_d)$ defined in the beginning of this subsection. Indeed in the above proof, we note that
\begin{align*}
    \langle e_0, D_i(\sum_{w\in [d]^*}\alpha_w e_w)&\rangle_{\FockqH}\\&=\sum_{m=1}(-1)^{m-1}q^{\frac{m(m-1)}{2}}\sum_{|v|=m-1}\left\langle \Bar{e}_v, \sum_{j=1}^d B_{ji}\tilde{L}_{vj}\left(\sum_{|w|=2m-1}\alpha_we_w\right)\right\rangle_{\FockqH}.
\end{align*}
Thus since for any word $v$ of length $m-1$ the subspace spanned by $\{\tilde{L}_{vj}(e_{w'})\}_{|w'|=2m'-1}$ is orthogonal to $\bar{e}_v$ for $m'\neq m$, 
  the expression for $D_i^*e_0$ (with $i \in \oned$) can  be given as follows:
\begin{align*}
    D_i^*e_0=\sum_{j=1}^d\sum_{m=1}^{\infty}\sum_{|v|=m-1}(-1)^{m-1}q^{\frac{m(m-1)}{2}} \bar{B}_{ji}\tilde{L}_{vj}^*(\Bar{e}_v), 
\end{align*}
where  the convergence of the sum follows exactly as in the proof of Proposition \ref{prop:conjugate}.
\end{rem}

In terminology of \cite{Brent} Proposition \ref{prop:conjugate} can be rephrased as saying that the set
$\{A_1,\dots, A_d\}$, which generates $\mlg$, has \emph{finite free Fisher information}. Together with Lemma \ref{lem:basechange} this yields the following corollary.

\begin{cor}\label{cor:freeFisher}
Let $\Hil_\br$ be a finite-dimensional real Hilbert  space equipped with an orthogonal group $(U_t)_{t \in \br}$. The algebra $\Gamma_q(\Hil_\br, U_t)$ equipped with the canonical state $\varphi$ is generated by a finite set $G=G^*$ of eigenoperators of the modular group of $\varphi$ with finite free Fisher information.
\end{cor}
\begin{proof}
	By \cite[Proof of Theorem 2.2]{Hiai} we can choose a set of linearly independent vectors $(\xi_1, \ldots, \xi_d)$ in $\Hil$ such that $W(\xi_1), \ldots, W(\xi_d)$ form a self-adjoint set of eigenoperators of the modular group of $\varphi$. Lemma \ref{lem:basechange}  and Proposition \ref{prop:conjugate} imply that 
	$\{W(\xi_1), \ldots, W(\xi_d)\}$ has finite free Fisher information (see \cite[Remark 3.13]{Brent}).
\end{proof}

\section{Consequences for structure of $q$-Araki-Woods von Neumann algebras} \label{sec:main}
We begin by quoting the main results of \cite{Brent} and some facts established in \cite{SW} (see also \cite{BM}).
\begin{tw}[\cite{Brent}, Theorem A]\label{Thm:Nelson}
Let $\mlg$ be a von Neumann algebra with a faithful normal state $\varphi$. Suppose $\mlg$ is generated by a finite set $G=G^{\ast}$, $|G|\geqslant 2$ of eigenoperators of the modular group $\sigma^{\varphi}$ with finite free Fisher information. Then $(\mlg^{\varphi})^{\prime} \cap \mlg = \mathbb{C}$. In particular, $\mlg^{\varphi}$ is a $\mathrm{II}_1$ factor and if $H < \mathbb{R}^{\times}_{\ast}$ is the closed subgroup generated by the eigenvalues of $G$ then $\mlg$ is a factor of type
\[
\left\{\begin{array}{ll} \mathrm{III}_1 & \text{ if } H=\mathbb{R}^{\times}_{\ast} \\ \mathrm{III}_{\lambda} & \text{ if } H= \lambda^{\mathbb{Z}}, 0<\lambda<1 \\ \mathrm{II}_1 & \text{ if } H=\{1\}.  \end{array}   \right.
\]
\end{tw}
\begin{tw}[\cite{Brent}, Theorem B]\label{Thm:Nelson2}
Let $\mlg$ be a von Neumann algebra with a faithful normal state $\varphi$. Suppose $\mlg$ is generated by a finite set $G=G^{\ast}$, $|G|\geqslant 2$ of eigenoperators of the modular group $\sigma^{\varphi}$ with finite free Fisher information. Then $\mlg^{\varphi}$ does not have property $\Gamma$. Furthermore, if $\mlg$ is a type $\mathrm{III}_{\lambda}$ factor, $0<\lambda<1$, then $\mlg$ is full.
\end{tw}

\begin{tw}[\cite{SW}, Lemma 5(2) with its proof, and Theorem 7(1)] \label{Thm:SW}
Let $(\Hil_{\mathbb{R}}, U_t) = (\mathsf{K}_{\mathbb{R}}, U_t^{\prime}) \oplus (\mathsf{L}_{\mathbb{R}}, U_{t}^{\prime\prime})$ be the decomposition into, respectively, the almost periodic and the weakly mixing part. Denote $\mlg:= \Gamma_q(\Hil_{\mathbb{R}}, U_t)$ and write $\mlg_1$ and $\mlg_2$ for the expected subalgebras corresponding to, respectively, the almost periodic and the weakly mixing parts. Then
\begin{enumerate}[{\normalfont (i) }]
\item $\mlg^{\varphi} \subset \mlg_1$, hence if $x\in \mlg\cap \mlg^{\prime}$ then $x \in \mlg_1$; 
\item if $(U_t)_{t \in \br}$ admits a non-zero fixed vector then $\mlg$ is a factor.
\end{enumerate}
\end{tw}
With Corollary \ref{cor:freeFisher} and these tools in hand we can completely characterize factoriality of $q$-Araki-Woods algebras and establish all the other results listed in the introduction. 
\begin{tw} \label{thm:factor}
Let $(\Hil_{\mathbb{R}}, U_t)$ be given, with $\dim(\Hil_{\br})\geqslant 2$. Then $\mlg:=\Gamma_q(\Hil_{\mathbb{R}}, U_t)$ is a factor. Moreover, if $G < \mathbb{R}^{\times}_{\ast}$ is the closed subgroup generated by the spectrum of $A$ then $\mlg$ is a factor of type
\[
\left\{\begin{array}{ll} \mathrm{III}_1 & \text{ if } G=\mathbb{R}^{\times}_{\ast} \\ \mathrm{III}_{\lambda} & \text{ if } G= \lambda^{\mathbb{Z}}, 0<\lambda<1 \\ \mathrm{II}_1 & \text{ if } G=\{1\}.  \end{array}   \right.
\]
\end{tw}
\begin{proof}
By Theorem \ref{Thm:SW} the center of $\mlg$ is contained in the almost periodic part, so we may assume it is nontrivial. If it is one dimensional, then it necessarily contains a non-zero $U_t$-invariant vector, so this case is covered by Theorem \ref{Thm:SW} as well. We can therefore assume that we are in the almost periodic case with $\dim(\Hil_{\br})\geqslant 2$. If  $\Hil_{\br}$ is infinite dimensional then factoriality has been obtained by Hiai (\cite[Theorem 3.2]{Hiai}, which in fact omits certain cases; see \cite[Theorem 4.3]{BMRW} for a complete result). In the finite dimensional case we can use   Corollary \ref{cor:freeFisher} and Theorem \ref{Thm:Nelson}. Note that the infinite-dimensional almost periodic case can be also deduced from the finite-dimensional one via the inductive limit argument.

If $(U_t)_{t \in \br}$ is almost periodic then the centralizer $\mlg^{\varphi}$ is irreducible in $\mlg$ (as follows from   Corollary \ref{cor:freeFisher} and Theorem \ref{Thm:Nelson} in finite dimensions and \cite[Theorem 3.2]{Hiai}, \cite[Theorem 5.1]{BMRW} in the infinite dimensional case). Therefore in this case the type classification can be simply obtained from the spectral data of $A$ as in the statement (see \cite[Section 1]{Hiai}). On the other hand, if there is a non-trivial weakly mixing part,  \cite[Theorem 8.1]{BM} implies that $\mlg$ is a $\mathrm{III}_1$ factor (see also \cite[Theorem 3.4]{Hiai} for an earlier result in the purely weakly mixing case).
\end{proof}

\begin{tw} \label{thm:noninj}
The factor $\Gamma_q(\Hil_{\br}, U_t)$ is not injective as soon as $\dim(\Hil_{\br}) \geqslant 2$.
\end{tw}

\begin{proof}
To prove non-injectivity, we will find an expected non-injective subalgebra; note that the case where the weakly mixing part is non-trivial has already been covered in \cite[Theorem 3.4]{Hiai}, where it was proved that in the purely weakly mixing case $\Gamma_q(\Hil_{\br}, U_t)$ is a non-injective factor.

We therefore assume that we are in the almost periodic case. It means that either we will find a two dimensional subspace on which $(U_t)_{t \in \br}$ is trivial, or a two dimensional subspace on which $(U_t)_{t \in \br}$ is ergodic. In both cases the corresponding $q$-Araki-Woods algebra will be non-injective and with expectation, from which we will be able to conclude. In the former we are just dealing with a $q$-Gaussian algebra, which was covered in \cite[Theorem 2]{Nou}. In the latter we can conclude from  Corollary \ref{cor:freeFisher} and Theorem \ref{Thm:Nelson2} that we have a type $\mathrm{III}_{\lambda}$ full subfactor, and fullness implies non-injectivity.
\end{proof}

We saw above that the $q$-Araki Woods factor is full when it is of type III$_\lambda$, $0<\lambda<1$ and  dimension of $\Hil_\br$ is finite. We now establish fullness in the remaining type III$_1$ finite-dimensional case as well. We also establish solidity of such factors (see \cite{Oza} for the original definition for finite von Neumann algebras and \cite{HR} for the modification needed in the general case).
We will first introduce a definition which will be also useful in the last part of the paper.

\begin{deft}
	Suppose that $\mlg$ is a von Neumann algebra (represented in a standard form on the Hilbert space $L^2(\mlg)$). We say that it satisfies the Akemann-Ostrand (AO) property if there exist  two weak$^*$-dense $C^*$-subalgebras $\alg \subset \mlg$ and $\blg \subset \mlg'$, with $\alg$ locally reflexive, such that the product-quotient map $\theta: \alg \odot \blg \to B(L^2(\mlg))/K(L^2(\mlg))$ is continuous with respect to the minimal tensor product norm. We say that a II$_1$-factor $\mlg$   satisfies the  $W^*$(AO) property if the product-quotient map $\tilde{\theta}: \mlg \odot \mlg' \to B(L^2(\mlg))/\mathbb{K}_\mlg$ is continuous with respect to the minimal tensor product norm, where $\mathbb{K}_\mlg$ is the $C^*$-subalgebra of $B(L^2(\mlg))$ defined in \cite{Taka}.
\end{deft}

Note that one could also define  the  $W^*$(AO) property for an arbitrary finite von Neumann algebra $\mlg$, but then one also needs to fix a faithful normal tracial state on $\mlg$ (as the definition of $\mathbb{K}_\mlg$ in \cite{Taka} formally refers to the chosen trace).

% Recall that a (separable) von Neumann algebra $M$ is {\em full} if for any bounded sequence $(x_n)$ in $M$ satisfying $\varphi([x_n,y])\to 0$ with respect to a normal faithful state $\varphi$, there is a bounded sequence $\lambda_n$ of scalars so that $x_n-\lambda_n\to0$.

%Let $\omega$ be a non-principal ultrafilter of $\mathbb{N}$ and  $M^\omega$ denote the ultraproduct of a von Neumann algebra $M$. We say that   a $M$  is $\omega$-solid if for any von Neumann subalgebra $Q$ of $M$ with expectation whose relative commutatnt $Q'\cap M^\omega$ is diffuse, then $Q$ is amenable. We also recall that if $A$ is a $C^*$-subalgebra of $M$ weakly dense in $M$, then $M$ satisfies 

\begin{tw} \label{thm:full}
Let $(\Hil_\br, U_t)$ be given with $2\leqslant\dim \Hil_\br<\infty$. Then $\mlg:=\Gamma_q(\Hil_\br, U_t)$ is solid and full. 
\end{tw}
\begin{proof}
If $(U_t)_{t \in \br}$ is trivial, then $\mlg$ is a $q$-Gaussian algebra and its fullness is proved in \cite{MS}. If $\mlg$ is of type III$_\lambda$, $0<\lambda<1$, the statement about fullness follows from  Corollary \ref{cor:freeFisher} and Theorem \ref{Thm:Nelson2}. 

It follows from \cite[Theorem 1.2]{Kuz} that the Cuntz-Toeplitz algebra $T_q(\Hil) \subset B(\FockqH)$, i.e.\ the $C^*$-algebra generated by the left creation operators $\{l^{\ast}_{\xi}: \xi \in \Hil\}$, is nuclear, as an extension of the Cuntz algebra by compacts. Arguing exactly as in \cite[Section 4]{Dima2} we can thus deduce that $\mlg$ satisfies the Akemann-Ostrand (AO) property. This further implies by \cite[Theorem A]{HR} that $\mlg$ is $\omega$-solid, where $\omega$ denotes a fixed non-principal ultrafilter and hence is solid (see the definition of $\omega$-solidity in \cite[Section 1]{HR}). Theorems \ref{Thm:Nelson} and \ref{Thm:Nelson2} together with   Corollary \ref{cor:freeFisher} imply  that the centralizer of $\mlg$  with respect to the canonical state  is a non-injective II$_1$ factor. We can thus invoke \cite[Proposition 3.10]{HR} to conclude fullness when $\mlg$ is a type III$_1$ factor.
\end{proof}

\begin{rem}
By \cite[Theorem 6.2]{HI} $q$-Araki-Woods factors are full if $(U_t)_{t \in \br}$  has a weakly mixing part. The same theorem also covers some almost periodic examples, but if the eigenvalues of the generator of $(U_t)_{t \in \br}$  grow sufficiently fast then fullness remains an open problem.
\end{rem}

Recall that in \cite[Corollary 3.3]{Cas} Caspers showed that the $q$-Gaussian von Neumann algebra  $\Gamma_q(\ell^2)$ is not isomorphic to $\Gamma_0(\ell^2)$ if $q \neq 0$, using the $W^*$(AO) property. We would like to finish the paper with formulating a tracial non-isomorphism result which can be proved along the similar lines, following \cite{Taka} and \cite{Cas}, and state a natural conjecture.

\begin{prop}\label{prop:AO}
	Suppose that $\mlg$ is a II$_1$-factor. Then if $\mlg$ satisfies the (AO) property, it also satisfies the $W^*$(AO) property.
\end{prop}
\begin{proof}
	It suffices to adopt the proofs of \cite[Theorem, Section 4]{Taka} and \cite[Theorem 2.2]{Cas}, which are formulated in the language of group von Neumann algebras, to the current situation. All the arguments remain valid.
\end{proof}

\begin{cor}
Suppose that $q, q' \in (-1,1)$, $q \neq 0$, and $\Hil_\br$ is a finite-dimensional real Hilbert space. Then the von Neumann algebras $\Gamma_q(\ell^2)$ and $\Gamma_{q'} (\Hil_\br)$ are non-isomorphic.
\end{cor}

\begin{proof}
We can assume that 
$\dim (\Hil_\br) \geqslant 2$. In \cite[Theorem 3.2]{Cas} Caspers shows that $\Gamma_q(\ell^2)$ does not have the $W^*$(AO) property. The arguments discussed in the proof of Theorem \ref{thm:full} show that  $\Gamma_{q'} (\Hil_\br)$ has the (AO) property; as by 
\cite{Eric} (or Theorem \ref{thm:factor} above)  $\Gamma_{q'} (\Hil_\br)$ is a finite factor, Proposition \ref{prop:AO} shows that it also has the $W^*$(AO) property. This conludes the proof.
\end{proof}

\begin{con}
Let $q, q' \in (-1,1)$, $ q \neq 0$, let $(U_t)_{t \in \br}$ be a group of orthogonal transformations of $\ell^2$, let $\Hil_\br$ be a finite-dimensional real Hilbert space, and let $(V_t)_{t \in \br}$ be a group of orthogonal transformations of $\Hil_\br$.
Then $q$-Araki-Woods von Neumann algebras $\Gamma_q(\ell^2, U_t)$	and $\Gamma_q(\Hil_\br, V_t)$	are non-isomorphic.
\end{con}

\smallskip

\noindent {\bf Acknowledgments. } 
A.S.\ and M.K.\ were  partially supported by the National Science Center (NCN) grant no. 2020/39/I/ST1/01566.
M.W.\ was  partially supported by the National Science Center (NCN) grant no. 2021/43/D/ST1/01446. The project is co-financed by the Polish National Agency for Academic
Exchange within Polish Returns Programme. 
We are grateful to Cyril Houdayer, Kunal Mukherjee and Simeng Wang for their comments on the first draft of this note.

\vspace{5 pt}
\includegraphics[scale=0.5]{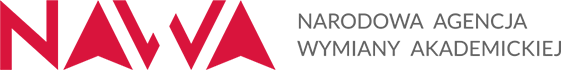}

%\smallskip

%\noindent {\bf Conflict of interest.} The authors have no conflicts of interest to declare that are relevant to the content of this article. 

\begin{comment}
\end{comment}

\end{document}